\documentclass[12pt]{article}
\usepackage{natbib}
\usepackage{amssymb,latexsym}
\usepackage{graphicx,amsmath,amsfonts,amsthm}
\usepackage{color}
\usepackage{graphicx}
\usepackage{etex}
\usepackage{morefloats}
\usepackage{float}

\input prepictex
\input postpictex
\input pictexwd

\newcommand{\R}{{\mathbb R}}
\newcommand\LL[1]{\multicolumn{1}{|c}{#1}}

\newcommand{\bpic}[4]{\beginpicture
  \setcoordinatesystem units  <1pt,1pt>
  \setplotarea x from #1 to #2, y from #3 to #4}
\newcommand{\epic}{\endpicture}
\newcommand{\hl}[3]{\put{\line(1,0){#1}} [Bl] at #2 #3 }
\newcommand{\vl}[3]{\put{\line(0,1){#1}} [Bl] at #2 #3 }

\newcommand{\bull}[2]{\put{$\bullet$} at #1 #2 }

\newtheorem{lemma}{Lemma}%[section]
\newtheorem{proposition}{Proposition}%[section]

\newtheorem{theorem}{Theorem}
\newtheorem{example}{Example}
\newtheorem{definition}{Definition}

\newcommand{\cd}{{\mathcal D}}

\newcommand{\cl}{{\mathcal L}}

\newcommand{\cc}{{\mathcal C}}

\def\row#1#2{{#1}_1,\ldots ,{#1}_{#2}}

\title{A combinatorial representation of Arrow's single-peaked domains}
\author{Arkadii Slinko}

\date{}

\begin{document}

\maketitle

\begin{abstract}
The most studied class of Condorcet domains (acyclic sets of linear orders) is the class of peak-pit domains of maximal width. It has a number of combinatorial representations by such familiar combinatorial objects like rhombus tilings and arrangements of pseudolines. Arrow's single-peaked domains are peak-pit but do not have maximal width. We suggest how to represent them by means of generalised arrangements of pseudolines.
\end{abstract}

\section{Introduction} 

Condorcet domain is a set of linear orders on a finite set of alternatives which do not lead to intransitivity if any multiset of these orders is aggregated using the majority voting. The theory of Condorcet domains is an old subject which has recently gained a new momentum summarised recently in \cite{puppe2024maximal}. All background information on Condorcet domains can be found in that article so we will not recap it here.\par\medskip

 An {\em arrangement of pseudolines}  is a family of pseudolines with the property that each pair of pseudolines has a unique point of intersection. An arrangement is {\em simple} \index{arrangement of pseudolines!simple} if no two distinct pairs of pseudolines intersect at the same point. In this paper the term {\it arrangement} will always denote a simple arrangement of pseudolines\footnote{more precisely simple labeled arrangement of pseudolines}. The {\it size} of an arrangement is the number of its pseudolines. Given an arrangement ${\bf a}$ of size $n$ we label its pseudolines $\row Ln$ with elements of $[n]$ so that they cross a vertical line  to the left of all intersections in increasing order from top to bottom and write ${\bf a}=\{\row Ln\}$. This vertical line will be denoted as $L$. Then the vertical line $R$ to the right of all intersections will be crossed in the reverse order (as every two lines must cross exactly once), that is in increasing order from bottom to top. The points of intersection of pseudolines are called {\em vertices}. The strip between the two lines $L$ and $R$ will be called {\em $LR$-strip} \index{LR-strip} and denoted by $E$. It is  endowed with the topology induced from $\R^2$. From now on a pseudoline is a piece of the graph of a continuous function connecting a point on $L$ with a point on $R$.

An arrangement ${\bf a}=\{\row Ln\}$ can be viewed as a cell complex $\cc({\bf a})$ with cells of dimensions $0$, $1$, or $2$, namely, the vertices, edges, and faces of the arrangement. The vertices are $L_i\cap L_j$ for $i\ne j$, the edges are connected components of $L_i\setminus \bigcup_{j\ne i} L_j$, $i=1,2,\ldots,n$, and the faces are connected components of $E\setminus \bigcup_{i=1}^n L_i$. Thus  $\cc({\bf a})=\bigcup_{i=0}^2\cc_i({\bf a})$, where $\cc_i({\bf a})$ is the set of cells of dimension $i$. The faces of this complex belonging to $\cc_2({\bf a})$ are usually called {\it chambers}. \index{chamber!} %The cells of an arrangement carry a natural lattice structure. %Adding a 0 and a 1 element we obtain the face lattice of the arrangement. 

The chambers are labeled by subsets of $[n]$, namely, a chamber $C$ gets label $\Lambda(C)=\{\row ik\}$ if lines $L_{i_1},\ldots, L_{i_k}$ and only they go above this chamber. These labels will be also called {\em chamber sets}. \index{chamber!set} Instead of writing labels as sets we will write them as strings, e.g., the aforementioned label will be written as $i_1i_2\ldots i_k$.  

\begin{definition}
We call a sequence $F=(F_0,F_1,\ldots,F_n)$ of subsets of $[n]$ a {\em flag} belonging to ${\bf a}$ if elements of $[n]$ can be ordered in a sequence $\row an$ so that $F_k=\{a_1,\ldots, a_k\}$. The linear order $a_1a_2\ldots a_n \in \cl([n])$ is said to {\em correspond} to the flag $F$.\index{linear order!corresponding to a flag}
\end{definition}

Thus, every arrangement of pseudolines ${\bf a}$ gives rise to the domain $\cd({\bf a})$ which consists of linear orders corresponding to all flags 
in ${\bf a}$. For example, the arrangement of pseudolines on Figure~\ref{fig:3-example} represents Condorcet domain $\cd_{3,1}=\{123, 213, 231, 321\}$.
\begin{figure}[H]
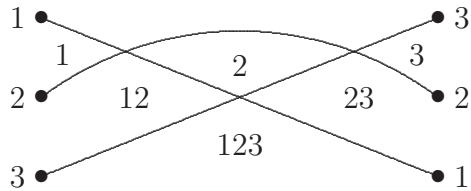

\centering
$$\bpic{-10}{160}{0}{60}
\bull{0}{0}
\put{$1$} at -9 60
\bull{0}{30}
\put{$2$} at -9 30
\bull{0}{60}
\put{$3$} at -9 0
\bull{150}{0}
\put{$1$} at 159 0
\bull{150}{30}
\put{$2$} at 159 30
\bull{150}{60}
\put{$3$} at 159 60
\circulararc 74 degrees from 150 30 center at 75 -70
\setlinear 
\plot 0 60 150 0 /
\plot 0 0 150 60 /
\put{$1$} at 8 46
\put{$3$} at 142 46
\put{$2$} at 75 42
\put{$12$} at 35 30
\put{$23$} at 120 30
\put{$123$} at 75 12
\epic$$
\caption{\label{fig:3-example} Pseudolines representation of domain $\cd_{3,1}$.}
\end{figure}

It appeared very beneficial to represent arrangements of pseudolines by the so-called {\em wiring diagrams} \citep{goodman1980proof} in which pseudolines are horizontal except the neighbourhood of each crossing. For example, the arrangement on Figure~\ref{fig:3-example} will be drawn as follows.

\begin{figure}[H]
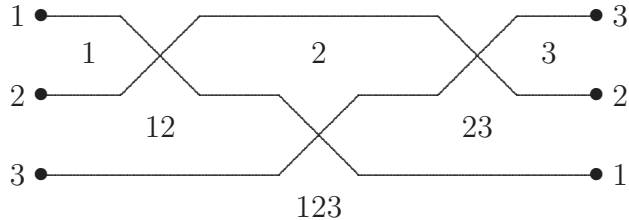

\centering
$$\bpic{-10}{160}{0}{60}
\bull{0}{0}
\put{$1$} at -9 60
\bull{0}{30}
\put{$2$} at -9 30
\bull{0}{60}
\put{$3$} at -9 0
\bull{210}{0}
\put{$1$} at 219 0
\bull{210}{30}
\put{$2$} at 219 30
\bull{210}{60}
\put{$3$} at 219 60
%\circulararc 74 degrees from 150 30 center at 75 -70
\setlinear 
\plot 0 60 30 60 60 30 90 30 120 0 210 0 /
%\plot 0 0 30 0 60 30 90 30 120 60 150 60 /
\plot 0 30 30 30 60 60 150 60 180 30 210 30 /
\plot 0 0 90 0 120 30 150 30 180 60 210 60 /
\put{$1$} at 18 46
\put{$3$} at 192 46
\put{$2$} at 105 46
\put{$12$} at 45 18
\put{$23$} at 165 18
\put{$123$} at 105 -12
\epic$$
\caption{\label{fig:3-example} Wiring diagram for domain $\cd_{3,1}$.}
\end{figure}
Intersections of pseudolines are called {\em vertices}.\par\medskip

Arrangements of pseudolines play a significant role in the theory of Condorcet domains.

\begin{theorem}
\label{thm:main-thm}
A domain $\cd$ is a peak-pit maximal Condorcet domain with maximal width if and only if it can be represented as $\cd({\bf a})$ for an arrangement of pseudolines ${\bf a}$.
\end{theorem}
This theorem follows from \cite{DKK:2012} where the representation of peak-pit maximal Condorcet domains with maximal width by rhombus tilings was proved and the equivalency of such a representation to the representation by arrangement of pseudolines (see \cite{elnitsky1997rhombic,felsner2012geometric} for more details. See also Corollary~3.1 in \cite{puppe2024maximal}.

However not all Condorcet domains have maximal width. The most prominent example is the class of Arrow's single-peaked domains (or in some papers locally single-peaked domains) which restrictions on all triples of alternatives are single-peaked but globally they may not be single-peaked. The reason is that the single-peakedness is defined relative to some societal axis which may exist for every triple but not for the whole domain. The structure of Arrow's single-peaked domains was investigated in \cite{slinko2019}.

The smallest maximal Arrow's single-peaked domain without maximal width is given in the following example.
 
\begin{example}
\label{ex:4-strange}
Let us consider a maximal Arrow's single-peaked domain on four alternatives:
\[
\cd_{4,2}= \{1234,\ 2134,\ 2314,\ 3214,\ 2341,\ 3241,\ 2431,\ 4231\},
\]
whose graph is presented on Figure~\ref{fig:4-strange}. %We have $\cd_{4,2}=\cd(\mathcal{N})$, where 
%\[
%\mathcal{N}=\{2N_{\{1,2,3\}}3,\  2N_{\{1,2,4\}}3,\  3N_{\{1,3,4\}}3,\  2N_{\{3,2,4\}}3 \}
%\]
%and $\cd_{4,2}$ is copious.  
This domain is not single-peaked (for example, because it does not have two completely reversed orders).

\begin{figure}[H]
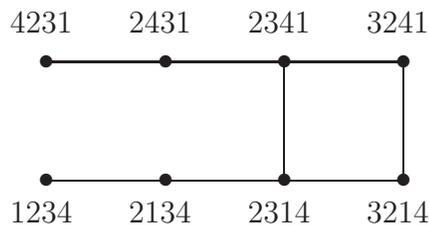

\centering
$$\bpic{-20}{180}{-20}{50}
\bull{0}{0}
\put{$1234$} at -2 -12
\hl{45}{0}{0}
\bull{45}{0}
\put{$2134$} at 43 -12
\put{$2314$} at 88 -12
\put{$3214$} at 133 -12
\bull{135}{45}
\bull{135}{0}
\bull{0}{45}
\bull{45}{45}
\vl{45}{135}{0}
\hl{90}{45}{0}
\bull{90}{0}
\vl{45}{90}{0}
\bull{90}{45}
\hl{45}{90}{45}
\hl{90}{0}{45}
\put{$2431$} at 43 60
\put{$3241$} at 133 60
\put{$2341$} at 88 60
\put{$4231$} at -2 60
\epic$$
\caption{\label{fig:4-strange} Graph of an Arrow's single-peaked domain $\cd_{4,2}$ which is not Black's single-peaked.}
\end{figure}
However, it has two extremal orders $1234$ and $4231$ where the terminal alternatives 1 and 4 are switched.
\end{example}

\iffalse
 Its ideal would be
% \[
  \begin{align*}
I_0&=\{\emptyset \},\\
I_1&=\{\{1\}, \{2\}, {\color{red}\{3\}}, \{4\}\},\\
I_2&=\{\{1,2\},\{2,3\}, {\color{red}\{2,4\}}\},\\
I_3&=\{\{1,2,3\}, \{2, 3,4\}\},\\
I_4&=\{\{1,2,3,4\}.
\end{align*}
%\]
%
We see it is weakly separated but not separated. \par\medskip
\fi

Let us try to represent the domain $\cd_{4,2}$ by arrangement of pseudolines.
The extremal orders of this domain are $1234$ and $4231$. %Thus we set $(i_1,i_2,i_3,i_4)=(1,2,3,4)$ and $(j_1,j_2,j_3,j_4)=(4,2,3,1)$. 
Drawing pseudolines in the classical way, we get the arrangement on Figure~\ref{fig:D42-}. This is not an arrangement of pseudolines in the classical sense since Line$_2$ and Line$_3$ do not intersect. As a result it represents only the subdomain $\cd'$ of $\cd_{4,2}$ on Figure~\ref{fig:D42-}:
\[
\cd'=\{1234, 2134, 2314, 2341, 2431,  4231\}.
\]
\begin{figure}[H]
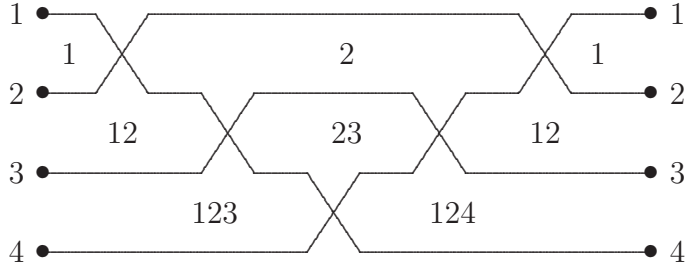

\centering
$$\bpic{-10}{280}{0}{50}
\bull{0}{0}
\bull{0}{30}
\bull{0}{60}
\bull{0}{-30}
\bull{230}{0}
\bull{230}{30}
\bull{230}{60}
\bull{230}{-30}
\setlinear 
\plot 0 60 20 60 40 30 60 30 80 0 100 0 120 -30 230 -30 /
\plot 0 30 20 30 40 60  180 60 200 30 230 30 /
\plot 0  0  60 0 80 30 140 30 160 0 230 0 / %270 60 /
\plot 0 -30 100 -30 120 0 140 0 160 30 180 30 200 60 230 60 /
\put{$4$} at -10 -30
\put{$3$} at -10 0
\put{$2$} at -10 30
\put{$1$} at -10 60
\put{$1$} at 240 60
\put{$2$} at 240 30
\put{$3$} at 240 0
\put{$4$} at 240 -30
\put{$1$} at 10 45
\put{$1$} at 210 45
\put{$2$} at 115 45
\put{$12$} at 30 15
\put{$23$} at 115 15
\put{$12$} at 190 15
\put{$123$} at 65 -15
\put{$124$} at 155 -15
\epic$$
\caption{\label{fig:D42-} Representation of a subdomain of $\cd_{4,2}$.}
\end{figure}

We need to do one trick to represent the whole $\cd_{4,2}$ which I discovered in 2019. Since we must have chamber $\{3\}$ the Line$_3$ must come to the top and hence must intersect Line$_2$. However since the order of 2 and 3 on $L$ and $R$ coincide,  Line$_3$ must intersect Line$_2$ twice.  And the pseudoline representation would be as shown on Figure~\ref{fig:D42}. It is easy to check that flags of such arrangement correspond to $\cd_{4,2}$.

%\begin{figure}[h]
%\begin{center}
%\includegraphics[height=5cm]{Pictures/pseudopseudolines.pdf}
%\end{center}
%\caption{Pseudoline representation of $\cd_{4,2}$.}
%\label{fig:with_double_crossing}
%\end{figure}

\begin{figure}[H]
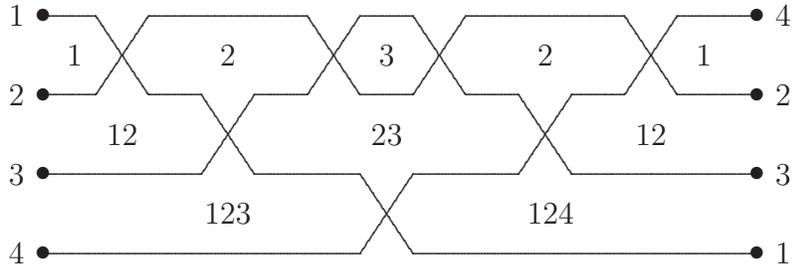

\centering
$$\bpic{-10}{280}{0}{50}
\bull{0}{0}
\bull{0}{30}
\bull{0}{60}
\bull{0}{-30}
\bull{270}{0}
\bull{270}{30}
\bull{270}{60}
\bull{270}{-30}
\setlinear 
\plot 0 60 20 60 40 30 60 30 80 0 120 0 140 -30 270 -30 /
\plot 0 30 20 30 40 60  100 60 120 30 140 30 160 60 220 60 240 30 270 30 /
\plot 0  0  60 0 80 30 100 30 120 60 140 60 160 30 180 30 200 0 230 0 270 0 /
\plot 0 -30 120 -30 140 0 180 0 200 30 220 30 240 60 270 60 /
\put{$4$} at -10 -30
\put{$3$} at -10 0
\put{$2$} at -10 30
\put{$1$} at -10 60
\put{$4$} at 280 60
\put{$2$} at 280 30
\put{$3$} at 280 0
\put{$1$} at 280 -30
%%%%%%%
\put{$1$} at 12 45
\put{$2$} at 70 45
\put{$2$} at 190 45
\put{$1$} at 250 45
\put{$3$} at 130 45
\put{$12$} at 30 15
\put{$23$} at 130 15
\put{$12$} at 230 15
\put{$123$} at 70 -15
\put{$124$} at 192 -15
\epic$$
\label{fig:D42}
\caption{Representation of domain of $\cd_{4,2}$.}
\end{figure}

\section{Main result}

Removing from the definition of arrangement of pseudolines the rquirement that any two lines intersect exactly once, we obtain a more general object which we will call {\em generalised arrangement of pseudolines}.
We ask: under which conditions the set of flags of a generalised arrangement of pseudolines is a Condorcet domain? \par\medskip

As in the classical case  a generalised arrangement of pseudolines ${\bf g}=\{\row Ln\}$ can be viewed as a cell complex $\cc({\bf g})$ with cells of dimensions $0$, $1$, or $2$, namely, the vertices, edges, and faces. The vertices are $L_i\cap L_j$ for $i\ne j$, the edges are connected components of $L_i\setminus \bigcup_{j\ne i} L_j$, $i=1,2,\ldots,n$, and the faces are connected components of $E\setminus \bigcup_{i=1}^n L_i$. Thus  $\cc({\bf g})=\bigcup_{i=0}^2\cc_i({\bf g})$, where $\cc_i({\bf g})$ is the set of cells of dimension $i$. The faces of this complex belonging to $\cc_2({\bf g})$ we will still call {\it chambers}. \index{chamber!} 

The chambers are labeled by subsets of $[n]$, namely, a chamber $C$ gets label $\Lambda(C)=\{\row ik\}$ if lines $L_{i_1},\ldots, L_{i_k}$ and only they go above this chamber. These labels will be also called {\em chamber sets}. \index{chamber!set} Instead of writing labels as sets we will write them as strings, e.g., the aforementioned label will be written as $i_1i_2\ldots i_k$.  

For any set of pseudolines ${\bf g}$ the collection of chamber sets $C_2({\bf g})$ is an ideal and $\cd({\bf g})=\text{Dom}(C_2({\bf g}))$ is a domain of linear orders. We need to find out when $\cd({\bf g})$ is a Condorcet domain. %For that we turn our attention to generalised arrangements of just three pseudolines. 
Firstly, we note that not for every generalised arrangement ${\bf g}$ the domain $\cd({\bf g})$ is a Condorcet domain. 

\begin{example}
\label{ex:non-Condorcet}
Let us consider the set ${\bf g}=\{L_i,L_j,L_k\}$ of three pseudolines on Figure~\ref{fig:three_bad_lines}. The corresponding domain
\[
\cd({\bf g})=\{ijk, jik, kij, kji \}
\]
is obviously not Condorcet.
%\begin{figure}[h]
%\begin{center}
%\includegraphics[height=3.3cm]{Pictures/three_bad_lines.png}
%\end{center}
%\caption{Non-Condorcet domain.}
%\label{fig:three_bad_lines}
%\end{figure}
%
\begin{figure}[H]
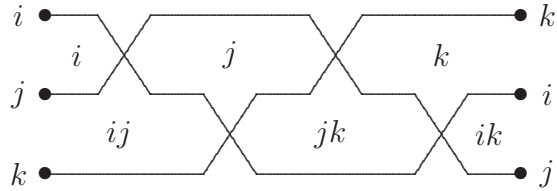

\centering
$$\bpic{-10}{200}{0}{60}
\bull{0}{0}
\bull{0}{30}
\bull{0}{60}
\bull{180}{0}
\bull{180}{30}
\bull{180}{60}
\setlinear 
\plot 0 60 20 60 40 30 60 30 80 0 140 0 160 30 180 30 / 
\plot 0 30 20 30 40 60 100 60 120 30 140 30 160 0 180 0 / 
\plot 0  0  60 0 80 30 100 30 120 60 140 60 180 60 / 
\put{$k$} at -10 0
\put{$j$} at -10 30
\put{$i$} at -10 60
\put{$k$} at 190 60
\put{$i$} at 190 30
\put{$j$} at 190 0
\put{$j$} at 70 45
\put{$k$} at 150 45
\put{$i$} at 12 45
\put{$ij$} at 28 15
\put{$jk$} at 108 15
\put{$ik$} at 168 15
\epic$$
\caption{Non-Condorcet domain.}
\label{fig:three_bad_lines}
\end{figure}
\end{example}
We note that Line$_i$ and Line$_j$ intersect twice at different levels. We will try to turn this observation into a criterion.

\begin{definition}
We call the arrangement of pseudolines {\em tame} if for any two lines, if they intersect more than once, all intersections are on the same level.
\end{definition}

\begin{lemma}
\label{lem:not-CD}
Let ${\bf g}$ be a generalised arrangement of $n$ pseudolines. Suppose there exist $i,j\in [n]$ such that Line$_i$ and Line$_j$ intersect twice at different levels.  Then $\cd({\bf g})$ is not Condorcet.
\end{lemma}

\begin{proof}
Let $u$ and $v$ be the two vertices at the intersections of Line$_i$ and Line$_j$. Since these two vertices are at different levels - we may assume $u$ is higher than $v$, - there must be a line Line$_k$ which goes  below $u$ but above $v$.
Let us restrict ${\bf g}$ to the subset of lines ${\bf g}'=\{L_i,L_j,L_k\}$. Then we obtain the configuration on Figure~\ref{fig:two vertices at different levels} or a similar configuration with $i$ and $j$ swapped.
\begin{figure}[H]
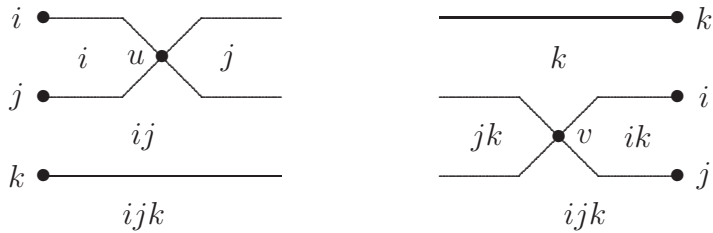

\centering
$$\bpic{-20}{280}{-20}{60}
\bull{0}{0}
\bull{0}{30}
\bull{0}{60}
\bull{240}{0}
\bull{240}{30}
\bull{240}{60}
\setlinear 
\plot 0 30 30 30 60 60 90 60 /
\plot 0 60 30 60 60 30 90 30 /
\bull{45}{45}
\put{$u$} at 35 45
\hl{90}{0}{0}
\put{$k$} at -10 0
\put{$j$} at -10 30
\put{$i$} at -10 60
\hl{90}{150}{60}
\plot 150 0 180 0 210 30 240 30 /
\plot 150 30 180 30 210 0 240 0 /
\bull{195}{15}
\put{$v$} at 205 15
\put{$k$} at 250 60
\put{$i$} at 250 30
\put{$j$} at 250 0
%%%%%%%%%%%%
\put{$i$} at 15 45
\put{$ij$} at 38 15
\put{$j$} at 70 45
\put{$k$} at 195 45
\put{$jk$} at 168 15
\put{$ik$} at 225 15
\put{$ijk$} at 38 -15
\put{$ijk$} at 205 -15
\epic$$
\label{fig:two vertices at different levels}
\caption{Line$_i$ and Line$_j$ intersect at two different levels.}
\end{figure}
Then $\cd({\bf g})$ contains orders $ijk, jik, kij, kji$ which cannot occur in a Condorcet domain. If $i$ and $j$ are swapped on $R$, then this will not affect our reasoning as only the chambers $ik$ and $jk$ will be swapped.
\end{proof}

\begin{theorem}
A generalised arrangement of $n$ pseudolines ${\bf g}$ has domain $\cd({\bf g})$ Condorcet if and only if it is tame. In such a case $\cd({\bf g})$ is peak-pit.
\end{theorem}

\begin{proof}
If ${\bf g}$ is not tame, by Lemma~\ref{lem:not-CD} $\cd({\bf g})$ is not Condorcet. Suppose it is tame. It is sufficient to consider the case when ${\bf g}$ has only three pseudolines. %It is easy to check  that the statement is true when $w({\bf g})\in {\cal C}$. In particular, $\cd({\bf g}(s_1s_2s_1))$ is copious Condorcet, satisfying $2N_{\{1,2,3\}}3$ and ${\cal P}(s_1s_2s_1)=(1,3)$. 

Let us prove that it is enough to consider generalised arrangements in which any two lines intersect at most two times. Suppose two lines, which without loss of generality can be taken as Line$_1$ and Line$_2$, intersect three times. As ${\bf g}$ is tame, all three intersections occur at the same level, say level 1.
\begin{figure}[H]
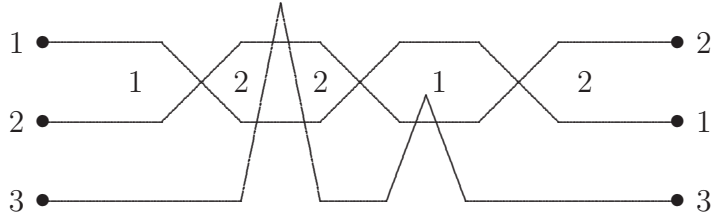

\centering
$$\bpic{-10}{280}{0}{60}
\bull{0}{0}
\bull{0}{30}
\bull{0}{60}
\bull{240}{0}
\bull{240}{30}
\bull{240}{60}
\setlinear 
\plot 0 30 45 30 75 60 105 60 135 30 165 30 195 60 240 60 /
\plot 0 60 45 60 75 30 105 30 135 60 165 60 195 30 240 30 /
\put{$1$} at 35 45
\put{$1$} at 150 45
\put{$2$} at 75 45
\put{$2$} at 105 45
\put{$2$} at 205 45
\plot 0 0 75 0 90 75 105 0 130 0 145 40 160 0 240 0 /
\put{$3$} at -10 0
\put{$2$} at -10 30
\put{$1$} at -10 60
\put{$2$} at 250 60
\put{$1$} at 250 30
\put{$3$} at 250 0
\epic$$
\label{fig:3ints}
\caption{Line$_1$ and Line$_2$ intersect three times. The third intersection can be removed.}
\end{figure}
Let us show that one of the intersections can be removed (straightened). For example let us look at chambers labelled by 2. If within some chamber Line$_3$ intersects both Line$_1$ and Line$_2$ (see Figure~\ref{fig:3ints}), then this chamber must be retained as it adds orders for $\cd({\bf g})$ with 3 on the top (otherwise $\cd({\bf g})$ satisfies $3N_{\{1,2,3\}}1$. Thus we will have orders 123, 213, 321, 231. If only line 3 penetrates the chamber labelled 1, as shown on Figure~\ref{fig:3ints}, then we will immediately have order with 2 is at the bottom and $\cd({\bf g})$ is not Condorcet. If not, this chamber can be removed by removing the double crossing creating it. 

Now let us show that, if no two lines in ${\bf g}$ intersect three times, the $\cd({\bf g})$ is peak-pit. This has already been done in  \cite{li2023classification} but we have a more direct way of showing this. If each pair of lines intersect once, this is the classical case in which we have only two arrangements of pseudolines. One is presented in Figure~\ref{fig:3-example}, the other is flip-isomorphic to it. One satisfies $2N_{\{1,2,3\}}3$ and the other one $2N_{\{1,2,3\}}1$. 

Suppose now that a certain pair of pseudolines intersect twice. Up to an isomorphism and flip-isomorphism we may consider that these two lines are Line$_1$ and Line$_2$ which intersect at vertices $u$ and $v$. If Line$_3$ does not intersect one of them, then ${\bf g}$ satisfies $3N_{\{1,2,3\}}1$. Line$_3$ can intersect Line$_1$ and Line$_2$ once before both $u$ and $v$. We then have generalised arrangement ${\bf g}$ shown on Figure~\ref{fig:one_cros} with $\cd({\bf g})$ satisfying $2N_{\{1,2,3\}}1$. 

The next possibility is for Line$_3$ to intersect both lines Line$_1$ and Line$_2$ twice before $u$ and $v$. We then have the following generalised arrangement ${\bf g}$ shown on Figure~\ref{fig:two_cros_before} with $\cd({\bf g})$ satisfying $1N_{\{1,2,3\}}3$.

\phantom\qedhere

The last possibility is for Line$_3$ to intersect both lines Line$_1$ and Line$_2$ twice between $u$ and $v$. We then have the following generalised arrangement ${\bf g}$ shown on Figure~\ref{fig:two_cros_between} with $\cd({\bf g})$ satisfying $2N_{\{1,2,3\}}3$.

This proves the theorem.
\begin{figure}[H]
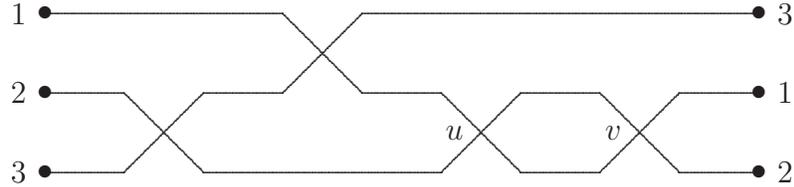

\centering
$$\bpic{-10}{280}{0}{60}
\bull{0}{0}
\bull{0}{30}
\bull{0}{60}
\bull{270}{0}
\bull{270}{30}
\bull{270}{60}
\setlinear 
\plot 0 60 90 60 120 30 150 30 180 0 210 0 240 30 270 30 /
\plot 0 30 30 30 60 0 150 0 180 30 210 30 240 0 270 0 /
\plot 0 0 30 0 60 30 90 30 120 60 270 60 /
\put{$3$} at -10 0
\put{$2$} at -10 30
\put{$1$} at -10 60
\put{$3$} at 280 60
\put{$1$} at 280 30
\put{$2$} at 280 0
\put{$u$} at 155 15
\put{$v$} at 215 15
\epic$$
\label{fig:one_cros}
\caption{Line$_1$ and Line$_2$ intersect two times. Line$_3$ intersects both of them once.}
\end{figure}

\begin{figure}[H]
\centering
$$\bpic{-10}{280}{0}{60}
\bull{0}{0}
\bull{0}{30}
\bull{0}{60}
\bull{270}{0}
\bull{270}{30}
\bull{270}{60}
\setlinear 
\plot 0 60 60 60 80 30 100 30 120 60 140 60 180 60 200 30 220 30 240 60 270 60 /
\plot 0 30 20 30 40 0 140 0 160 30 180 30 200 60 220 60 240 30 270 30 /
\plot 0  0  20 0 40 30 60 30 80 60 100 60 120 30 140 30 160 0 270 0 / %270 60 /
\put{$3$} at -10 0
\put{$2$} at -10 30
\put{$1$} at -10 60
\put{$1$} at 280 60
\put{$2$} at 280 30
\put{$3$} at 280 0
\put{$u$} at 182 45
\put{$v$} at 222 45
\epic$$
\label{fig:two_cros_before}
\caption{Line$_1$ and Line$_2$ intersect two times. Line$_3$ intersects both of them twice before $u$ and $v$.}
\end{figure}

\begin{figure}[H]
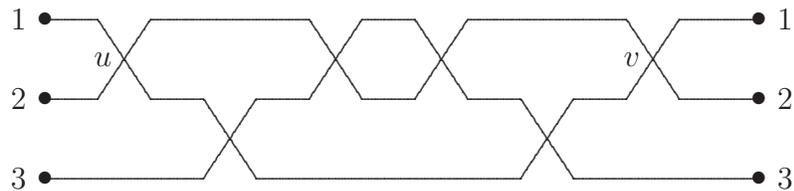

\centering
$$\bpic{-10}{280}{0}{60}
\bull{0}{0}
\bull{0}{30}
\bull{0}{60}
\bull{270}{0}
\bull{270}{30}
\bull{270}{60}
\setlinear 
\plot 0 60 20 60 40 30 60 30 80 0 180 0 200 30 220 30 240 60 270 60 /
\plot 0 30 20 30 40 60 100 60 120 30 140 30 160 60 220 60 240 30 270 30 /
\plot 0  0  60 0 80 30 100 30 120 60 140 60 160 30 180 30 200 0 270 0 / %270 60 /
\put{$3$} at -10 0
\put{$2$} at -10 30
\put{$1$} at -10 60
\put{$1$} at 280 60
\put{$2$} at 280 30
\put{$3$} at 280 0
\put{$u$} at 22 45
\put{$v$} at 222 45
\epic$$
\label{fig:two_cros_between}
\caption{Line$_1$ and Line$_2$ intersect two times. Line$_3$ intersects both of them twice between $u$ and $v$.}
\end{figure}
\end{proof}

\begin{theorem}
Any maximal Arrow's single-peaked Condorcet domain is representable as the domain $\cd({\bf g})$ associated with a generalised arrangement of pseudolines ${\bf g}$.
\end{theorem}

\begin{proof}
Let $\cd\subseteq \cl([n])$ be an Arrow's single-peaked domain. We know that $\cd$ has two terminal alternatives,  which can be taken to be 1 and $n$, and two extremal linear orders, where terminal alternatives occupy the top and bottom positions.  

We know that $\cd_{-1}$ and $\cd_{-n}$ which are the restrictions of $\cd$ onto $[n-1]$ and $[n]\setminus \{1\}$, are two maximal Arrow's single-peaked domains on $n-1$ alternatives. Moreover $\mathcal{E}=(\cd_{-1})_{-n}=(\cd_{-n})_{-1}$ is a maximal Condorcet domain on $[n]\setminus \{1,n\}$ common for $\cd_{-1}$ and $\cd_{-n}$. We note that to move from $\cd_{-1}$ to $\cd_{-n}$ we need only one switch of 1 and $n$.

\begin{equation}
\label{relay}
\left[\begin{array}{ccccccc|ccccccc}
1&\ &\ &&\LL{}&&& &&&\LL{}&\ &\ &\  n \\
&\ &\ &&\LL{}&&&&& &\LL{}&\ &\ &\  \\
&\ &\ &&\LL{}&\mathcal{E}&\ &&\ \mathcal{E}&&\LL{}&&\ &\  \\
&\ &\ &&\LL{}&&&&&& \LL{}&\ &\ &\ \\
&\ &\ &&\LL{}&&&&&& \LL{}&\ &\ &\ \\
   \cline{5-10}
&\ &\ &&\LL{1}&\cdots&1& n&\cdots&n&\LL{}&\ \\
  \cline{1-14}
n&&&&\cdots&& n&1 & &\cdots&&& &1 
\end{array}\right].
\end{equation}

By the induction hypothesis each of $\cd_{-1}$ and $\cd_{-n}$ can be represented as $\cd({\bf g}_{-1})$ and $\cd({\bf g}_{-n})$, respectively, on sets of pseudolines ${\bf g}_{-1}$ and ${\bf g}_{-n}$, respectively. Now we create ${\bf g}$ representing $\cd$ as shown on Figure~\ref{fig:g-1+g-n}. 
\begin{figure}[H]
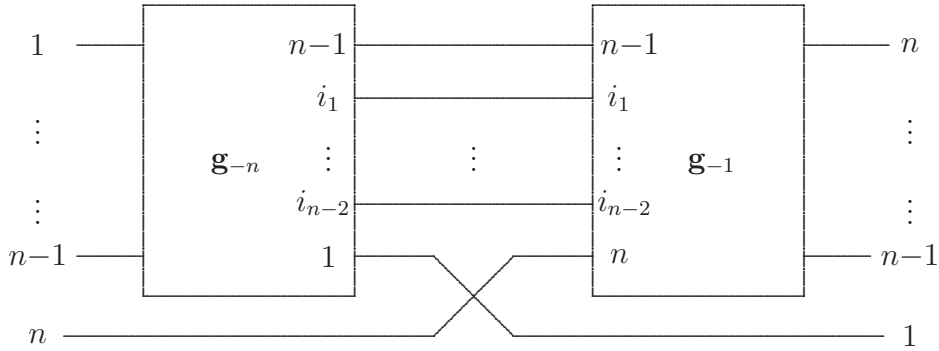

\centering
$$\bpic{-20}{310}{0}{80}
\setlinear 
\plot 0 0 140 0 170 30 200 30 / 
\plot 110 30 140 30 170 0 310 0 / 
\put{$\vdots$} at 155 70
\plot 110 90 200 90 / 
\put{$i_1$} at 100 90
\put{$\vdots$} at 100 70
\put{$i_{n{-}2}$} at 98 50
\put{$i_1$} at 210 90
\put{$\vdots$} at 210 70
\put{$i_{n{-}2}$} at 212 50
\put{$n{-}1$} at 214 110
\plot 110 50 200 50 / 
\plot 110 110 200 110 / 
\put{$n$} at -10 0
\put{$n{-}1$} at -10 30
\put{$\vdots$} at -10 50
\put{$\vdots$} at -10 80
\put{$1$} at -10 110
\put{$1$} at 320 0
\put{$n{-}1$} at 320 30
\put{$\vdots$} at 320 50
\put{$\vdots$} at 320 80
\put{$n$} at 320 110
\put{$1$} at 100 30
\put{$n$} at 210 30
\put{$n{-}1$} at 96 110
\plot 5 30 30 30 /
\put{${\bf g}_{-n}$} at 65 65
\put{${\bf g}_{-1}$} at 245 65
\plot 5 110 30 110 /
\plot 30 15 110 15 110 125 30 125 30 15 / 
\plot 200 15 280 15 280 125  200 125 200 15 / 
\plot 280 110 312 110 /
\plot 280 30 305 30 /
\epic$$
\caption{\label{fig:g-1+g-n} Method of combining ${\bf g}_{-1}$ and ${\bf g}_{-n}$.}
\end{figure}
Line$_n$ is added to ${\bf g}_{-n}$ trivially staying all the time at the bottom of the arrangement---this reflects $n$ being the bottom alternative in \eqref{relay}. Similarly Line$_1$ is added to ${\bf g}_{-1}$. Then these lines swap before entering ${\bf g}_{-1}$ and ${\bf g}_{-n}$, respectively.
\end{proof}

\section{Relation to the literature}

This paper primarily related to the paper by \cite{li2023classification} which unfortunately has several fatal mistakes to which \cite{Puppe-Slinko24} pointed out. We have certain differences in definitions with that papers too. Firstly, they have the narrower concept of generalised arrangements of pseudolines restricting the number of intersections for any two pseudoloines to at most two. They showed in Example~4, however, that this concept does not work as intended demonstrating that one Arrow's single-peaked Condorcet domain on five alternatives cannot be represented.  Due to this they had to modify the definition of a flag making it more general. In that paper they claim to have proved that  with the modified concept of a flag any peak-pit domain (not only Arrow's single-peaked) is representable by a generalised arrangement of pseudolines.

\bibliographystyle{plainnat}
\bibliography{cps}
\end{document}